\documentclass[12pt,a4paper,oneside]{amsart}

\usepackage{amsfonts, amsmath, amssymb, amsthm, amscd}
\usepackage{hyperref}
\hypersetup{
  colorlinks   = true, 
  urlcolor     = blue, 
  linkcolor    = blue, 
  citecolor   = red 
}
\usepackage{anysize}

\newtheorem{theorem}{Theorem}[section]

\theoremstyle{definition}
\newtheorem{definition}[theorem]{Definition}
\newtheorem{property}[theorem]{Property}

\theoremstyle{remark}
\newtheorem{remark}[theorem]{Remark}

\numberwithin{equation}{section}

\newcommand{\cyc}{\mathit{cl}}

\renewcommand{\int}{\mathop{\rm int}}

\renewcommand{\epsilon}{\varepsilon}

\begin{document}

\title{Dependence of the heavily covered point on parameters}

\author{Alexey~Balitskiy{$^{\spadesuit}$}}

\email{alexey\_m39@mail.ru}

\author{Roman~Karasev{$^{\clubsuit}$}}

\email{r\_n\_karasev@mail.ru}
\urladdr{http://www.rkarasev.ru/en/}

\address{{$^{\spadesuit}$}{$^{\clubsuit}$}Moscow Institute of Physics and Technology, Institutskiy per. 9, Dolgoprudny, Russia 141700}
\address{{$^{\spadesuit}$}{$^{\clubsuit}$}Institute for Information Transmission Problems RAS, Bolshoy Karetny per. 19, Moscow, Russia 127994}

\thanks{{$^\clubsuit$} Supported by the Dynasty foundation.}
\thanks{{$^\spadesuit$}{$^\clubsuit$} Supported by the Russian Foundation for Basic Research grant 15-31-20403 (mol\_a\_ved).}
\thanks{{$^\spadesuit$}{$^\clubsuit$} Supported by the Russian Foundation for Basic Research grant 15-01-99563 A}


\begin{abstract}
We examine Gromov's method of selecting a point ``heavily covered'' by simplices formed by a given finite point sets, in order to understand the dependence of the heavily covered point on parameters. We have no continuous dependence, but manage to utilize the ``homological continuous dependence'' of the heavily covered point. This allows us to infer some corollaries in a usual way. We also give an elementary argument to prove the simplest of these corollaries.
\end{abstract}

\maketitle

\section{Introduction}

Boros, F\"uredi, and B\'ar\'any in~\cite{bf1984,ba1982} established a theorem on existence of a point ``heavily covered'' by simplices: For a prescribed $N$-point set $X$ in $\mathbb R^d$, it is possible to find a point $c\in\mathbb R^d$ that is covered by an essential fraction $c_d>0$ of all the $\binom{N}{d+1}$ simplices that can be chosen with vertices in $X$. In another paper~\cite{pach1998} Pach established a similar result in a ``colorful'' version, for selecting the vertices $\{x_0,\ldots, x_d\}$ of the simplices in all possible ways from their respective sets $X_0, \ldots, X_d$ and finding a point that is covered by a certain fraction of all those simplices.

The next important step was made recently, in~\cite{grom2010} Gromov has developed a certain topological approach based on the ideas of ``bounded cohomology'', and proved the following generalization of the ``heavily covered point'' results:

\begin{theorem}[Gromov, 2010]
\label{theorem:gromov}
For any continuous map $f$ from the $(N-1)$-dimensional simplex $\Delta$ to the Euclidean space $\mathbb R^d$, some point $c\in \mathbb R^d$ is covered by a fraction of at least $\frac{2d}{(d+1)!(d+1)}$ of all images of $d$-dimensional faces of $\Delta$.
\end{theorem}

\begin{remark}
In fact, the fraction was $\frac{2d}{(d+1)!(d+1)} - O(1/N)$, but we ignore the correcting term throughout.
\end{remark}

The classical case follows from this theorem when $f$ is chosen to be the linear map taking the vertices of $\Delta$ to the corresponding points of $X$. Gromov's approach allowed not only to generalize the result, but also to improve the known bounds for the fraction $c_d$ to $\frac{2d}{(d+1)!(d+1)}$. Soon after further improvements were made using Gromov's method, see~\cite{matwag2011,kms2012}, for example.

One thing that could facilitate further applications and generalizations of Theorem~\ref{theorem:gromov} would be the continuous dependence of the point $c$ (somehow selected from the heavily covered points) on the map $f$. This is the question that we consider in this note.

In fact, it is not likely to expect any continuous behavior, but instead we want to show the ``homologically continuous'' behavior. This behavior is very typical in those cases, when the solution of the problem is guaranteed (co)homologically. Namely, we are going to present a construction of an odd \emph{$0$-cycle modulo $2$} in $\mathbb R^d$, that is an odd number of points in $\mathbb R^d$, any of which is heavily covered and such that this $0$-cycle modulo $2$ depends on $f$ continuously.

The continuous dependence of a $0$-cycle modulo $2$ is informally understood in the following way (see also~\cite{akv2012}): Either the points of the cycle move continuously with the change of parameters, or a pair (or several pairs) of them annihilate at the same position, or some pair gets born from nothing at some position.

A more geometric way is to treat an odd $0$-cycle modulo $2$ in $\mathbb R^d$ continuously parameterized by a manifold $P$ to be a subset $C\subset P\times \mathbb R^d$ that is an embedded modulo 2 pseudomanifold of codimension $d$ such that its projection onto $P$ is a proper map of degree $1$ modulo $2$. This kind of continuous dependence is sufficient to apply common methods of generalizing the result by taking a Cartesian product with some other geometric construction. Here a \emph{modulo 2 pseudomanifold of dimension $d$} is a $d$-dimensional simplicial complex such that every its $(d-1)$-dimensional face is contained in precisely two $d$-dimensional faces.

Certain consequences of this ``homological continuity'' are given in Section~\ref{section:applications} and Section~\ref{section:elementary} contains more elementary observations in the planar case using some simple topology.

\medskip
\textbf{Acknowledgments. }
We thank Benjamin Matschke for useful remarks and stimulating questions, and Sergey Melikhov for pointing out the useful review~\cite{fried2008}.

\section{Proof of homological continuity}
\label{section:proof}

Before proving the homologically continuous dependence of $c$ on $f$ we recall Gromov's proof. Though there are simpler proofs for Theorem~\ref{theorem:gromov}, like in~\cite{kar2011}, we need the original Gromov's argument here.

\subsection{Space of cocycles}
The key idea of Gromov is considering the space of $d$-dimensional cocycles modulo $2$ in $\Delta$, denoted by $\cyc^d(\Delta; \mathbb F_2)$. This space is constructed in~\cite{grom2010} combinatorially in the following manner: Let $X$ be any topological space, take its cochain complex (let us work with modulo $2$ coefficients and omit $\mathbb F_2$ from the notation)
$$
\begin{CD}
0 @>>> C^0(X) @>{d}>> C^1(X) @>>> \dots @>>> C^d(X) @>>> \dots.
\end{CD}
$$
Then we build a certain simplicial abelian group out of this. Its $k$-simplices are the chain maps (reversing the degree): $\sigma : C_*(\Delta^k) \to C^{d-*}(X)$. These objects constitute a functor from the simplex category to abelian groups (maps) and therefore define a simplicial abelian group, which is natural to call the space of $d$-cocycles of $X$, $\cyc^d(X; \mathbb F_2)$. Speaking informally, the points of this space are $d$-cocycles themselves, the edges between the points correspond to the ``homologous'' relation, and the higher-dimensional faces are attached to them according to the chain complex.

Since $\cyc^d(X; \mathbb F_2)$ is naturally a simplicial abelian group, it is a Kan/fibrant complex (see the nice introductory text~\cite{fried2008} or the textbook~\cite{weibel1994}) and its homotopy groups can be calculated formally as $\pi_k(\cyc^d(X; \mathbb F_2)) = H^{d-k}(X; \mathbb F_2)$ (meaning that the simplicial maps of a simplex $\Delta^k$ to $\cyc^d(X; \mathbb F_2)$ up to homotopy coincide with the cohomology group of $X$), this is what Gromov~\cite{grom2010} calls the Dold--Thom--Almgren theorem. It might be not very clear why it makes sense to call this object a ``space of $d$-cocycles'', but intuitively it is very useful to consider it as such, in particular the Dold--Thom theorem for the other version of the cycle space built from currents holds in the same manner, as was shown by Almgren~\cite{alm1962}. Actually, we are also going to consider such ``spaces of cycles'' built from the chain complex of a simplicial complex $X$.

In fact, the above mentioned properties of $\cyc^d(X; \mathbb F_2)$ are consequences of the following:

\begin{property}
\label{property:chain-map}
For a simplicial complex $W$ a simplicial map $W\to \cyc^d(X; \mathbb F_2)$ (in the category of simplicial sets) is the same as a chain map $C_*(W; \mathbb F_2) \to C^{d-*}(X; \mathbb F_2)$.
\end{property}

For any connected $X$, in the $d$-dimensional cohomology of $\cyc^d(X; \mathbb F_2)$, $H^d(\cyc^d(X; \mathbb F_2); \mathbb F_2)$, there exists a certain canonical class $\xi$. Its existence can be guessed from the ``Dold--Thom--Almgren theorem'', or it can be built explicitly by counting how many times a vertex (as a $0$-cocycle) of $X$ participates in a $d$-cycle of cocycles from $\cyc^d(X; \mathbb F_2)$.

More precisely, using Property~\ref{property:chain-map} we view a $d$-cycle of cocycles as a simplicial map $c : P\to \cyc^d(X; \mathbb F_2)$, where $P$ is a modulo $2$ pseudomanifold of dimension $d$, or equivalently, a chain map $C_*(P;\mathbb F_2) \to C^{d-*}(X; \mathbb F_2)$.

\begin{definition}
\label{definition:can-class}
The value of $\xi$ on the cycle of $P$ is the image of the $d$-dimensional fundamental class $[P]$ in $H^0(X; \mathbb F_2) = \mathbb F_2$.
\end{definition}

\subsection{Start of Gromov's proof: The ``inverse map''}

Return to our particular space $X = \Delta^N$, or just $\Delta$, and build the space of cocycles $\mathcal C^d = \cyc^d(\Delta; \mathbb F_2)$ from the chain complex of $\Delta$, viewed as a simplicial complex.

Now, in our problem, we compactify the target space $\mathbb R^d$ with the point at infinity (denoted by $\infty$) to obtain the sphere $S^d$. So $f : \Delta\to \mathbb R^d$ becomes a map to $S^d$ not touching $\infty$. Any point $y\in S^d$ defines a $d$-cocycle on $\Delta$ by counting the multiplicity (modulo $2$) on the $f$-image of every $d$-face of $\Delta$ at $y$. In a certain sense this cocycle has continuous dependence on $y$ and therefore there arises a continuous map $f^c : S^d \to \mathcal C^d$, mapping the point $\infty$ to the zero $d$-cocycle.

More formally, using Property~\ref{property:chain-map}, to describe $f^c$ we have to consider a triangulation $T$ of $S^d$ and build the chain map $C_*(T; \mathbb F_2) \to C^{d-*}(\Delta; \mathbb F_2)$, defined by counting (modulo 2) intersections of faces of $T$ and faces of $f(\Delta)$ of complementary dimension. The crucial fact, noted in~\cite{grom2010} is that the canonical class $\xi$ evaluates on the image $f^c(S^d)$ to $1$, thus showing that the map $f^c$ is homologically nontrivial. Indeed, the chain map $C_*(T; \mathbb F_2) \to C^{d-*}(\Delta; \mathbb F_2)$ sends the fundamental class $[S^d]$ to the cochain in $C^{0}(\Delta; \mathbb F_2)$ having coefficient $1$ at every vertex of $\Delta$ (since in the general position every vertex of $\Delta$ is mapped by $f$ into precisely one $d$-face of $T$), and we apply Definition~\ref{definition:can-class}.

\subsection{More details: subspace of thin cocycles}

The other step of the proof is to note a different thing, which we are going to explain. We consider the subspace $\mathcal C^d_0\subset \mathcal C^d$, corresponding to those cocycles in $C^d(\Delta; \mathbb F_2)$ whose supports invoke less than the fraction $\frac{2d}{(d+1)!(d+1)}$ of $d$-faces of $\Delta$. We start with a very informal observation that $\mathcal C^d_0$ can be contracted to the zero cycle $0\in \mathcal C^d$ by a certain cocycle filling process. This observation (in~\cite{grom2010}) concludes the proof of Theorem~\ref{theorem:gromov}, because $f^c$ must touch $\mathcal C^d\setminus \mathcal C^d_0$ in order to be homologically nontrivial (as the cohomology class $\xi$ shows).

Now we add more details to the above sketch. In the process of contraction, in fact, there was another assumption: that the cycles of lower dimension $d-k$ in $\Delta$ constituting the cycle $f^c(T)$ (the same as simplices of positive dimension $k$ of $f^c(T)$) must have negligible complexity, that is consist of a tiny $\varepsilon$-fraction of all $(d-k)$-simplices of $\Delta$ for positive $k$. In~\cite{grom2010} this assumption was satisfied by taking sufficiently fine triangulation $T$ of $S^d$, then $f^c(T)$ has this property. So we have to include this assumption directly into the definition of the subspace $\mathcal C^d_0$, to define the subspace $\mathcal C^d_\varepsilon\subset \mathcal C^d$. More precisely, for $\mathcal C^d_\varepsilon$ we require that its faces, considered as chain maps $\sigma : C_*(\Delta^k) \to C^{d-*}(\Delta)$, are such that their $(d-k)$-dimensional parts invoke less that the $\varepsilon$-fraction of all $(d-k)$-faces of $\Delta$ in their support for $k>0$, and the fraction of used $d$-faces of $\Delta$ is at most the magic constant $\frac{2d}{(d+1)!(d+1)}$, possibly up to some correction term of order $1/N$ when $N\to \infty$. In fact, our first naive definition of $\mathcal C^d_0$ expressed this set as a collection of vertices of $\mathcal C^d$, so we really have to use $\mathcal C^d_\varepsilon$ for sufficiently small $\varepsilon$ to make this subspace contractible and the whole argument valid.

So what we need from the subspace $\mathcal C^d_\varepsilon$ is the following:

\begin{property}
\label{property:contraction}
Every $d$-dimensional modulo 2 cycle in $\mathcal C^d_\varepsilon$, expressed as a map of a modulo 2 pseudomanifold $P$ to $\mathcal C^d_\varepsilon$, can be contracted to a single point inside $\mathcal C^d$.
\end{property}

This property is proved in~\cite{grom2010} by extending the map $P\to \mathcal C^d_\varepsilon$ to the cone over $P$ using the ``linear filling profile'' in the complex $C_*(\Delta)$ (remember Property~\ref{property:chain-map}). This linear filling profile (actually having unit norm) allows one to invert the boundary operator $\partial : C_*(\Delta) \to C_*(\Delta)$ economically. This inversion process is iterated over the skeleton of $P$, and finally assures that a chain arising in $C_0(\Delta)$ is actually a cycle because of having a small norm (the size of its support), thus extending the map over a $(d+1)$-face of the cone over $P$.

The next step is to note that, when we take a sufficiently fine triangulation $T$ of $S^d$ in our argument, the image $f^c(T)$ gets into $\mathcal C^d_\varepsilon$, provided its ``vertex part'' is in $\mathcal C^d_0$. Property~\ref{property:contraction} means that the canonical class $\xi$ vanishes when restricted to the subspace $\mathcal C^d_\varepsilon\subset \mathcal C^d$ and therefore (by the exact sequence of the pair) it can be represented by a cocycle $X\in C^d(\mathcal C^d; \mathbb F_2)$ with support outside $\mathcal C^d_\varepsilon$.

So we assume that the canonical class $\xi$ is represented by a suitable cocycle $X$ not touching $\mathcal C^d_\varepsilon$. Then the pullback of this canonical cocycle under the map $f^c : S^d \to  \mathcal C^d$ makes a cocycle $X_f \in C^d(S^d; \mathbb F_2)$. In order for the pullback to be defined on the level of cocycles, we assume a sufficiently fine triangulation $T$ of $S^d$ (to fit into $\mathcal C^d_\varepsilon$ from the previous paragraph) and then $X_f$ becomes a $d$-cocycle assigning $0$ or $1$ to every $d$-face of $T$, and having the following property: The support of $X_f$ in $S^d$ consists of some $d$ faces of $T$, each of which having a ``heavily covered'' point as its vertex. Here a ``heavily covered'' point is a point covered by at least $\frac{2d}{(d+1)!(d+1)}$ of all images of $d$-faces of $\Delta$.

\subsection{Homologically continuous dependence}

Now we observe that $f^c(S^d)$ depends continuously on $f$ in some intuitive sense, and therefore the cocycle $X_f$ ``cut by $X$'' on $f^c(S^d)$ depends continuously on $f$, assuming certain topology on the space of cocycles, like the one implicit in the above construction.

The continuous dependence can be understood more precisely as follows: Let everything depend on a parameter space $P$, which we assume to be a PL-manifold; so $\tilde f : \Delta\times P \to S^d\times P$ is a family of maps over $P$. Consider $S^d\times P$ as having sufficiently fine triangulation. Every $k$-face $F\subseteq \Delta$ makes a $(k\dim P)$-dimensional subset $F_P = \tilde f(F\times P)$ in $S^d\times P$. Then for every vertex $v$ of $S^d\times P$ we count $d$-faces of $\Delta$ whose $F_P$ cover $v$ and obtain a cochain in $C^d(\Delta; \mathbb F_2)$ for every vertex. For every $1$-face $\sigma$ of $S^d\times P$ we count $(d-1)$-faces $F$ of $\Delta$ whose $F_P$ intersect $\sigma$ and obtain the respective cochains in $C^{d-1}(\Delta; \mathbb F_2)$, and so on with $k$-faces of $S^d\times P$ and cochains in $C^{d-k}(\Delta; \mathbb F_2)$. Those cochains are arranged in a proper way and give a chain map from the complex $C_*(S^d\times P; \mathbb F_2)$ to $C^{d-*}(\Delta; \mathbb F_2)$, which by Property~\ref{property:chain-map} can be viewed as a simplicial map from $S^d\times P$ to $\mathcal C^d$, this is just the definition of $\mathcal C^d$ as a simplicial abelian group (see~\cite{grom2010}). Therefore our $f^c$ gets extended to a simplicial map from $S^d\times P$, which is a simplicial version of the notion of ``a family of maps depending on the parameter $p\in P$ continuously''.

In order to understand the pullback $X_f$, we consider $X$ as a $d$-cocycle on $S^d\times P$ and make a fiberwise (along $S^d$) Poincar\'e duality, making a cycle $X'\in C_{\dim P} (S^d\times P; \mathbb F_2)$ such that its projection $X'\to P$ onto the parameter space has degree 1 modulo 2. Recall that we assume $P$ to be a PL-manifold for simplicity. This $X'$ is therefore interpreted as a family of $0$-cycles in $\cyc_0(S^d; \mathbb F_2)$ parameterized by $P$, whose all points represent heavily covered points of $S^d$ under their respective maps $f_p$ for $p\in P$.

\section{Applications}
\label{section:applications}

Having established the ``homologically continuous'' dependence we readily infer the following ``transversal'' and ``dual'' analogues of the results in~\cite{dol1993,zivvre1990} or \cite{kar2008}:

\begin{theorem}
\label{theorem:transversal}
Let $P_0, \ldots, P_m$ be finite point sets in $\mathbb R^d$, then there exists an $m$-dimensional affine plane $L$ such that for any $i=0,\ldots, m$ the following holds: The fraction of those $(d-m+1)$-tuples in $P_i$ whose convex hulls touch $L$ is at least
$$
\frac{2(d-m)}{(d-m+1)!(d-m+1)}.
$$
\end{theorem}

\begin{theorem}
\label{theorem:dual}
Let $\mathcal H$ be a family of hyperplanes in $\mathbb R^d$ in general position. Then there exists a point $c\in \mathbb R^d$ such that the fraction of those $(d+1)$-tuples of hyperplanes in $\mathcal H$ that surround $c$ is at least $\frac{2d}{(d+1)!(d+1)}$.
\end{theorem}

In the last theorem a family of $d+1$ hyperplanes \emph{surround} $c$ if $c$ cannot be moved to $\infty$ without touching any of these hyperplanes.

\begin{remark}
The proof with a worse constant $\frac{1}{(d+1)^d}$ in~\cite{ba1982} uses the Tverberg theorem to find the point $c$. Since there are topological analogues of the Tverberg theorem that have a similar homological continuity modulo a prime $p$, Theorems~\ref{theorem:transversal} and \ref{theorem:dual} with a weaker constant follow immediately. In fact, the constant will be further spoiled because of the requirement that the number of parts in a Tverberg partition is a prime power $p^\alpha$, and by the limitation to even $d-m$ in Theorem~\ref{theorem:transversal} for odd primes $p$, like in~\cite{kar2007}. In other words, the versions of  Theorems~\ref{theorem:transversal} and \ref{theorem:dual} with worse constant directly follow from the corresponding versions of the Tverberg theorem in~\cite{kar2007,kar2008} using the argument from~\cite{ba1982}.
\end{remark}

\begin{proof}[Proof of Theorem~\ref{theorem:transversal}]
Choose arbitrary $m$-dimensional linear subspace $V$ of $\mathbb R^d$. Then we apply Theorem~\ref{theorem:gromov} (for linear maps $f_i$, mapping the corresponding simplices linearly following mappings of the vertices to the projected $P_i$'s) to the projections of $P_i$'s onto the orthogonal complement $V^\perp$. We observe the corresponding cycles $c_i(V)$ of heavily covered points. What we need is to find a nonempty intersection of the supports of $c_i(V)$ for some $V$.

Now observe that all possible choices of $V$ and $V^\perp$ constitute the Grassmannian $G_{d, d-m}$, and the family of all possible $V^\perp$ gives the canonical vector bundle $\gamma = \gamma_{d, d-m}$ over $G_{d, d-m}$. It is known (see~\cite{dol1993,zivvre1990}) that the $m$-th power of the Euler class of $\gamma$ is nonzero modulo $2$. Equivalently, the $(m+1)$th power of the Thom class $\tau(\gamma)$ is nonzero modulo $2$ in the cohomology of the Thom space $M(\gamma)$ (the compactified total space of $\gamma$). This makes any $m+1$ sections of $\gamma$ coincide over some $V\in G_{d, d-m}$.

But this also makes any $m+1$ ``homological sections'' coincide somewhere in $\gamma$. Indeed, the homological sections $c_i(V)$ are modulo 2 pseudomanifolds that project with degree 1 mod 2 onto the base. Hence they are all Poincar\'e dual to the Thom class in the Thom space, and hence their intersection must be nontrivial. This gives the desired $V$ and intersections $c_i$'s, thus completing the proof.

In fact, the last part can be done in less geometric terms: Note that $c_i(V)$'s were defined in Section~\ref{section:proof} as $(d-m)$-cocycles on the compactification of $V^\perp$. And the parameterized version of Gromov's argument means that with varying $V$ those cocycles constitute a single $(d-m)$-cocycle of the Thom space $M(\gamma)$, representing the Thom class in the cohomology of $M(\gamma)$. Since the $(m+1)$th power of the Thom class does not vanish, their supports must have a common point.
\end{proof}

\begin{proof}[Proof of Theorem~\ref{theorem:dual}]
We act like in~\cite{kar2008} (see its corrected \href{http://arxiv.org/abs/0909.4915}{arxiv.org version}): Take a convex body $B$ so that every orthogonal projection $\pi_H$ onto a hyperplane $H\in\mathcal H$ takes $B$ into itself. Then for every point $x\in B$ we consider the point set $P_x = \{\pi_H(x)\}_{\mathcal H}$. We can find a point $c$ such that at least $\frac{2d}{(d+1)!(d+1)}$ of all $d$-simplices with vertices in $P_x$ contain $c$, moreover, by the ``homological continuous'' dependence we can choose a $0$-dimensional cycle $c(x)$ of such points depending continuously on $x$. In other words, this cycle is a multivalued map whose graph $\Gamma_c\subset B\times B$ is a modulo $2$ pseudomanifold with boundary projecting onto the first summand $B$ with degree $1$ modulo $2$ and taking boundary to boundary.

Then we follow the usual proof of the Brouwer fixed point theorem by showing that $\Gamma_c$ intersects the diagonal of $B\times B$, for example, by writing down the classes of $\Gamma_c$ and the diagonal $\Delta_B\subset B\times B$ in the relative $d$-dimensional homology of $(B,\partial B)\times (B, \partial B)$. We deduce that for some $x\in B$ some point in the support of $c(x)$ coincides with $x$, then, using the general position assumption, we conclude that those $d$-simplices of points $\pi_H(x)$ that contain $x$ correspond to those $(d+1)$-tuples of hyperplanes in $\mathcal H$ that surround $x$.
\end{proof}

\begin{remark}
Using an appropriate modification of Gromov's technique, like in~\cite{kar2011} with the improvement of the constant in~\cite{jiang2014}, one can prove the ``colorful'' version of the above theorems. That is, in Theorem~\ref{theorem:transversal} $P_i$ can be given colored into $d-m+1$ colors each so that every color covers exactly the fraction $\frac{1}{d-m+1}$ of its corresponding set $P_i$, and in the conclusion $L$ will touch at least the fraction $\frac{2(d-m)}{(d-m+1)!(d-m+1)}$ of all ``rainbow simplices'' (having no repetition of colors) of $P_i$.

Similarly, in Theorem~\ref{theorem:dual} $\mathcal H$ may be given colored in $d+1$ colors uniformly and in the conclusion the fraction of $(d+1)$-tuples of hyperplanes surrounding $c$ in all possible ``rainbow $(d+1)$-tuples'' is shown to be at least $\frac{2d}{(d+1)!(d+1)}$. The last claim may be considered as one step of another approach to the result of~\cite{barpa2013}.
\end{remark}

\begin{remark}
The improvements of the constant $\frac{2d}{(d+1)!(d+1)}$ from~\cite{matwag2011,kms2012} are applicable in the above theorems as well, since they actually improve the ``filling process'' used to contract $d$-cycles in what we call $\mathcal C^d_\varepsilon$.
\end{remark}

\section{Elementary observations about the planar case}
\label{section:elementary}

It is known that for $d=2$ Theorem~\ref{theorem:gromov} has very elementary proofs, see~\cite{bukh2006} or \cite{fglnp2012}, for example. Here we are going to mimick the one in~\cite{fglnp2012} to prove the dual version:

\begin{proof}[Elementary proof of the planar case of Theorem~\ref{theorem:dual}]
Let us have a set of lines in $\mathbb R^2$. When we are going to speak about a random line, or a random pair of lines, or a random triple of lines, we will choose the lines from the given set uniformly.

Now we consider pairs of a halfplane $H$ and a point $q\in \partial H$. We will also parameterize the halfplane $H$ by its inner unit normal $p$, considered as a vector at $q$. Such a pair $(q, H)$ is called \emph{exposed} if the probability that a pair of lines $\ell_1$ and $\ell_2$ cuts from $H$ a triangle having $q$ on its base is less than $2/9$.

We will usually translate the question about lines to questions about points, having fixed the point $q$, by replacing a line $\ell$ with the projection of $q$ onto $\ell$, let $q(\ell) = \pi_\ell(q)$. In terms of points $q_1 = q(\ell_1)$ and $q_2=q(\ell_2)$ the condition of \emph{exposed} is expressed as follows: The probability that the random segment $q_1q_2$ intersects the ray $\{q + tp\}_{t\ge 0}$ is less than $2/9$; this is almost the same of the definition of ``exposed'' in~\cite{fglnp2012}.

Now we want to find a point $q$ such that no pair $(q, p)$ is exposed, for any $p$. This point will be the required point in the theorem, since for any fixed choice of the first line $\ell_1$ in the triple, the event ``the triple $\{\ell_1, \ell_2, \ell_3\}$ surrounds $q$'' is equivalent to the event ``the pair $\{\ell_2, \ell_3\}$ cuts from $H_{\ell_1}$ a triangle having $q$ on its base''. Here $H_{\ell_1}$ is the halfplane with $\partial H_{\ell_1}$ parallel to $\ell_1$, containing $q$, and disjoint from $\ell_1$. If none of $(q, H_{\ell_i})$ is exposed then we sum the probabilities to have the probability of ``the triple $\{\ell_1, \ell_2, \ell_3\}$ surrounds $q$'' to be at least $2/9$.

So we assume that no such $q$ exists. Now we also follow the proof in~\cite{fglnp2012}. When two pairs $(q, p_1)$ and $(q, p_2)$ with the same $q$ are exposed, then we consider the rays $r_1 = \{q + tp_1\}_{t\ge 0}$ and $r_2 = \{q + tp_2\}_{t\ge 0}$. The rays divide the plane into two parts $P_1$ and $P_2$ and one of the parts must have $< 1/3$ of the points $q(\ell)$. Otherwise the probability that the two points $\{q(\ell_1), q(\ell_2)\}$ are in different $P_i$'s will be at least $4/9$, and for one of the rays $r_i$, the probability that the segment $[q(\ell_1), q(\ell_2)]$ intersects $r_i$ is at least $2/9$, which contradicts the definition of ``exposed''. So we conclude that one of $P_i$'s contains less than $1/3$ of the points $q(\ell)$. For any $p$, whose ray $r = \{q + tp\}_{t\ge 0}$ is in this $P_i$, we call the pair $(q, p)$ \emph{almost exposed}.

As in~\cite{fglnp2012}, it is easy to conclude that for any $q\in\mathbb R^2$ the set $F_q$ of $p$ such that the pair $(q, p)$ is almost exposed is a nonempty proper connected subset of the circle $S^1$. It is nonempty because of our assumption; and it is proper because for an exposed pair $(q,p_1)$ we can always find another non-exposed pair $(q, p_2)$. It is achieved, when the rays $r_1 = \{q + tp_1\}_{t\ge 0}$ and $r_2 = \{q + tp_2\}_{t\ge 0}$ partition the set $\{q(\ell)\}$ into equal parts, this argument is also from~\cite{fglnp2012}.

Now the proof is finished by applying an appropriate version of the Brouwer fixed point theorem for convex-valued maps with closed graphs to the map $q\mapsto S^1\setminus F_q$, as in~\cite{fglnp2012}.
\end{proof}

Now we give an example showing that the constant $2/9$ is the best possible in the planar case of Theorem~\ref{theorem:dual}. Consider a circle $C$ and choose a sequence of points $x_1,\ldots, x_n$ in the given order on an arc of $C$ of angular measure at most $\pi/2$. Let $\ell_1,\ldots,\ell_n$ be the lines tangent to $C$ at their respective points $x_i$'s.

How a point $q$ can be surrounded by some three of $\ell_i$'s? It is easy to observe that $q$ must lie outside $C$ and the three surrounding lines must have three types: One line $\ell_{i_2}$ must separate $q$ from $C$, another line $\ell_{i_1}$ must have $q$ and $C$ on the same side of it and $q$ must be to the ``right'' of $C$ (if we place the arc $[x_1x_n]$ approximately horizontally), and the third line $\ell_{i_3}$ must also have $q$ and $C$ on the same side of it and $q$ must be to the ``left'' of $C$.

In fact, for any given $q$ outside $C$ all the lines are separated into such three classes of cardinalities $n_1,n_2,n_3$ so that $n_1+n_2+n_3 = n$. So the number of triangles surrounding $q$ is always bounded by
$$
n_1n_2n_3 \le \frac{(n_1+n_2+n_3)^3}{27} = \frac{n^3}{27},
$$
which approaches $\frac{2}{9}\binom{n}{3}$ for large $n$.

\bibliography{../bib/karasev}
\bibliographystyle{abbrv}
\end{document}